\documentclass[11pt,reqno]{amsart}
\usepackage{amsmath,amssymb,amsthm}
\usepackage{geometry}

\usepackage{booktabs}
\usepackage{microtype}
\usepackage[hidelinks]{hyperref}
\hypersetup{pdftitle={Skew CNS and sharp Ax--Katz congruences},pdfauthor={}}
\numberwithin{equation}{section}
\newtheorem{theorem}{Theorem}[section]
\newtheorem{proposition}[theorem]{Proposition}
\newtheorem{lemma}[theorem]{Lemma}
\newtheorem{corollary}[theorem]{Corollary}
\theoremstyle{definition}

\newtheorem{example}[theorem]{Example}
\theoremstyle{remark}
\newtheorem{remark}[theorem]{Remark}
\newcommand{\T}{\mathcal T}
\newcommand{\rk}{\operatorname{rk}}
\newcommand{\F}{\mathbb F}

\newcommand{\vp}{\nu_p}
\newcommand{\Span}{\operatorname{span}}
\newcommand{\Nm}{\operatorname{N}}

\newcommand{\ceil}[1]{\left\lceil #1\right\rceil}

\newcommand{\supp}{\operatorname{supp}}
\newcommand{\Ccount}{\mathcal C}
\newcommand{\sigmaspace}{\Sigma_n}
\title[Skew CNS and sharp Ax--Katz congruences]{From skew Combinatorial Nullstellensatz\\to sharp Ax--Katz congruences and zero counts}
\date{26 October 2026}
\author{Shuang Li}
\author{Xuan Pang}
\author{Pingzhi Yuan}
\hypersetup{pdfauthor={Xuan Pang, Shuang Li, Pingzhi Yuan}}
\begin{document}
\raggedbottom
\begin{abstract}
We give a unified account of zero counting in finite-field polynomial rings with commuting variables twisted by a common automorphism. A decomposition into fixed-field slices connects skew evaluation with ordinary polynomial evaluation. First, we derive a slice coefficient formula from the skew Combinatorial Nullstellensatz and use it to deduce skew Chevalley--Warning, including the sufficient condition $n>h\sum_i d_i$. Next, classical Ax--Katz gives the corrected integer congruence $Q^{\max\{0,\lceil(n-h\sum_i d_i)/D\rceil\}}\mid N+(M-1)\delta$. We construct homogeneous examples proving optimality for every positive degree table and every number of variables. We also establish a sharp lower bound with a prescribed zero, an input-sensitive coefficient-filtration refinement, and transfers of known bounds for nonlinear zero sets. Detailed calculations explain the origin correction, incomplete norm blocks, and the distinction between characteristic-$p$ summation and higher integer divisibility. Existing geometric, interpolation, and counting theorems are explicitly attributed.
\end{abstract}
\maketitle

\section{Introduction}

The number of common zeros of low-degree polynomials over a finite field is constrained by the Chevalley--Warning and Ax--Katz theorems. In a skew polynomial ring, the evaluation and the ambient set of points must first be specified. We work in the setting of Alon, Behajaina and Paran \cite{ABP}: the variables commute with one another and act on the coefficient field through a common automorphism.

Throughout the paper, let
\begin{equation}\label{eq:fields}
 K=\F_Q,\qquad Q=p^a,\qquad F=\F_{Q^h},\qquad
 K=F^\sigma,
\end{equation}
where $p$ is prime, $q=Q^h=|F|$, and $\sigma$ has order $h$. Thus $\sigma$ may be any generator of $\operatorname{Gal}(F/K)$; it need not be the $Q$-power Frobenius. Let
\[
 R=F[x_1,\ldots,x_n;\sigma],\qquad x_i c=\sigma(c)x_i,
 \qquad x_i x_j=x_j x_i.
\]
We assume $n,r\ge1$. Each polynomial is written with its coefficients on the left. We consider nonzero polynomials $f_1,\ldots,f_r\in R$ of positive total degrees $d_1,\ldots,d_r$. Zero equations may be deleted, and a nonzero constant equation gives an empty zero set. These conventions account separately for systems not satisfying the positive-degree assumption.

Put
\begin{equation}\label{eq:parameters}
 M=\frac{Q^h-1}{Q-1},\qquad E=h\sum_{i=1}^r d_i,
 \qquad D=\max_i d_i,\qquad
 \mu=\max\left\{0,\ceil{\frac{n-E}{D}}\right\}.
\end{equation}
Write $\vp$ for the $p$-adic valuation of integers, with $\vp(0)=+\infty$. The compatible space $\sigmaspace$ is defined in Section~\ref{sec:slices}. If $N$ is the number of common skew zeros on this space and $\delta\in\{0,1\}$ indicates whether the origin is a common zero, the principal congruence is
\begin{equation}\label{eq:introcongruence}
 Q^\mu\mid\Ccount,\qquad \Ccount=N+(M-1)\delta.
\end{equation}
The correction in $\Ccount$ is forced by the common origin of the fixed-field slices. For $h>1$, $\delta=1$ and $\mu\ge2$, it gives the exact obstruction $\vp(N)=a$ to an uncorrected higher congruence.

The sufficient degree condition in \cite[Theorem~4.5]{ABP}, in our notation, is $n>M\sum_i d_i$, and its conclusion is $p\mid N$. We replace it by $n>h\sum_i d_i$ and obtain $Q\mid N$. Since $M>h$ for $h>1$, this improves that condition. Norm-form examples show that the new threshold is best possible using only the degree table. This settles the degree-bound question in the first part of \cite[Problem~6.1]{ABP}; the separate request for a deduction from skew CNS is treated in Section~\ref{c:sec:deduction}.

The congruence, its sharpness, and the lower bound proved below are consequences in the skew setting. The geometric decomposition used to derive them is already in \cite[Lemmas~4.1--4.2]{ABP}. The classical divisibility theorem is due to Ax and Katz \cite{Ax,Katz}; the classical optimality for arbitrary degree tables is \cite[Proposition~4.0]{Katz}. We give a direct norm-form construction as well. Warning's lower bound and the estimates of Leep and Petrik are cited explicitly where used. The coefficient-filtration argument concerns optimization under linear combinations of scalar equations, rather than the exact valuation of every individual input.

The paper develops one counting method at two different levels of precision. At the first level, skew CNS leads to a sum that vanishes in a field of characteristic $p$. This proves divisibility of an integer count by $p$. At the second level, classical Ax--Katz controls the integer count itself and supplies powers of $Q$. These are related routes, but the second conclusion cannot be read off from a field-valued sum. Keeping this distinction explicit prevents a common mistake when $Q$ is not prime.

The common starting point is geometric. A compatible point has all its nonzero coordinate ratios in the fixed field. Consequently the compatible space is a collection of $K$-grids, all passing through the origin. A linear change of coefficients converts skew evaluation on any one of these grids into ordinary evaluation on $K^n$. Once this change is made, the zero-counting problem becomes a collection of ordinary finite-field problems. The correction term in the count is the exact bookkeeping required when these grids are put back together.

\subsection{Notation and organization}
The following list records the symbols used throughout the paper. The field size $Q$ and the extension degree $h$ play different roles: $Q$ controls ordinary finite-field counting, whereas $h$ measures the number of scalar coordinates needed to express an $F$-valued equation over $K$.
\begin{center}
\begin{tabular}{ll}
\toprule
Symbol & Meaning\\
\midrule
$p,a,Q$ & $p$ prime, $Q=p^a=|K|$\\
$h,q$ & $[F:K]=h$, $q=Q^h=|F|$\\
$M$ & Number of fixed-field slices, $(q-1)/(Q-1)$\\
$d_i,D$ & Prescribed degrees, and their maximum\\
$E$ & Scalar degree budget $h\sum_i d_i$\\
$Z_\omega$ & Ordinary zero count on a single slice\\
$N,\delta$ & Skew zero count, and the origin indicator\\
$\mathcal C$ & Corrected integer count $N+(M-1)\delta$\\
\bottomrule
\end{tabular}
\end{center}
Section~\ref{sec:slices} explains evaluation before any counting theorem is used. Sections~\ref{c:sec:interpolation} and \ref{c:sec:deduction} give the CNS route, including its multiplication obstruction. Section~\ref{sec:uniform} then introduces the stronger classical inputs and proves the integer congruence. Section~\ref{sec:sharpness} constructs examples and verifies every residue case needed for sharpness. The later sections replace the coarse degree budget by information from coefficients, transfer established nonlinear bounds, and work through concrete examples. The appendices supply a short alternative CNS argument and an elementary proof of the classical lower bound.

\section{Compatible points and fixed-field slices}\label{sec:slices}

For $b=(b_1,\ldots,b_n)\in F^n$, define
\begin{equation}\label{eq:compatible}
 \sigmaspace=\{b\in F^n:\sigma(b_i)b_j=\sigma(b_j)b_i
 \text{ for all }i,j\}.
\end{equation}
This is the $\sigma$-affine space of \cite{ABP}. For completeness, the evaluation on this space can be described using a left $R$-module. At a compatible point $b$, let $c\in F$ act on $F$ by multiplication and let
\[
 x_i\cdot t=\sigma(t)b_i\qquad(t\in F).
\]
Let $X_i(t)=\sigma(t)b_i$. For $c\in F$ we have
\[
 X_i(ct)=\sigma(c)\sigma(t)b_i=\sigma(c)X_i(t),
\]
which checks the coefficient relation. Composing two variable operators gives
\[
 X_iX_j(t)=\sigma^2(t)\sigma(b_j)b_i,\qquad
 X_jX_i(t)=\sigma^2(t)\sigma(b_i)b_j.
\]
Thus they commute exactly when the compatibility equations hold. These calculations explain why the point space is restricted before evaluation is defined. Define
\[
 f(b)=f\cdot1.
\]
This is the evaluation by right remainders used in \cite{ABP}. Indeed, the left ideal $\sum_i R(x_i-b_i)$ annihilates $1$; reduction by its generators shows that its quotient is spanned by $1$ over $F$, and the module just defined shows that no nonzero scalar lies in this ideal. In particular, the remainder is unique. Polynomial evaluation here is not asserted to be a ring homomorphism.

Choose a set $\Omega\subset F^\times$ of representatives of $F^\times/K^\times$, and define
\begin{equation}\label{eq:twisting}
 A_0(\omega)=1,\qquad
 A_d(\omega)=\prod_{j=0}^{d-1}\sigma^j(\omega)\quad(d\ge1).
\end{equation}
If $f=\sum_{d\ge0}f_d$ is its homogeneous decomposition, write
\begin{equation}\label{eq:transform}
 f^{[\omega]}(u)=\sum_{d\ge0}A_d(\omega)f_d(u)
 \in F[u_1,\ldots,u_n],
\end{equation}
where the variables on the right are ordinary commuting variables.

\begin{lemma}[Fixed-field slices; cf.\ \cite{ABP}]\label{lem:slices}
The sets $\omega K^n$, for $\omega\in\Omega$, cover $\sigmaspace$, and two distinct such sets intersect only at the origin. Moreover,
\begin{equation}\label{eq:evaluation}
 f(\omega u)=f^{[\omega]}(u)\qquad(u\in K^n).
\end{equation}
\end{lemma}
\begin{proof}
If $0\ne b\in\sigmaspace$ and $b_j\ne0$, then
\[
 \sigma(b_i)b_j=\sigma(b_j)b_i
 \quad\Longrightarrow\quad
 \frac{\sigma(b_i)}{\sigma(b_j)}=\frac{b_i}{b_j}.
\]
The left side is $\sigma(b_i/b_j)$, so $b_i/b_j\in K$ for every $i$. Hence $b\in b_jK^n$. The converse follows by substitution in \eqref{eq:compatible}. If $\omega u=\eta v\ne0$ with $u,v\in K^n$, a nonzero coordinate gives $\omega/\eta\in K^\times$, proving the intersection assertion.

For $b=\omega u$, the module action on a monomial of total degree $d$ contributes the successive factors $\omega,\sigma(\omega),\ldots,\sigma^{d-1}(\omega)$. All coordinates of $u$ are fixed by $\sigma$. Therefore $x^\alpha(\omega u)=A_{|\alpha|}(\omega)u^\alpha$. Left $F$-linearity proves \eqref{eq:evaluation}.
\end{proof}

For each $\omega\in\Omega$, let
\begin{equation}\label{eq:slicecount}
 Z_\omega=\#\{u\in K^n:f_i^{[\omega]}(u)=0\text{ for all }i\}.
\end{equation}
The map $u\mapsto\omega u$ is a bijection onto its slice. Its origin is a common zero exactly when $\delta=1$. Thus Lemma~\ref{lem:slices} gives the integer identity
\begin{equation}\label{eq:countidentity}
 \sum_{\omega\in\Omega} Z_\omega=N+(M-1)\delta=\Ccount.
\end{equation}
Unlike its reduction modulo $p$, this identity retains the multiplicity of the common origin. In particular, $|\sigmaspace|=1+M(Q^n-1)$. The counts are independent of the choice of representatives: replacing $\omega$ by $\omega c$, with $c\in K^\times$, replaces $f^{[\omega]}(u)$ by $f^{[\omega]}(cu)$.

\subsection{The coefficient transform as a linear map}
It is useful to name the operation already appearing in \eqref{eq:transform}. Write
\[
 \T_\omega(f)=f^{[\omega]},
\]
using ordinary variables $t$ or $u$ as convenient. On a monomial of degree $d$, this map multiplies the coefficient by the nonzero scalar $A_d(\omega)$. It therefore has an inverse, obtained by dividing that coefficient by $A_d(\omega)$. Both maps preserve the support and every degree bound. They are $F$-linear maps of coefficient spaces.

The inverse map is particularly useful for a polynomial built by ordinary multiplication: first form the ordinary polynomial, and only then apply $\T_\omega^{-1}$. Its skew evaluation on $\omega K^n$ reproduces the desired ordinary function. This observation will let us build zero indicators without assuming a false multiplication rule for skew evaluation.

For example, a quadratic monomial evaluates as
\[
 (x_ix_j)(\omega u)=\sigma(\omega)\omega u_iu_j.
\]
Its ordinary product of coordinate values would instead be $\omega^2u_iu_j$. These expressions need not agree. The coefficient transform records precisely the extra factor required by the skew action.

\subsection{Why the origin changes integer divisibility}
Each map $u\mapsto\omega u$ is injective, so a single slice causes no overcounting. Overcounting appears only when different slices are combined. If the origin is not a zero, no correction is needed. If it is a zero, every slice includes that same point. The sum of $M$ slice counts includes the origin $M$ times, whereas $N$ includes it once; this explains $(M-1)\delta$ in \eqref{eq:countidentity}.

There are two useful arithmetic facts about the correction. Always $Q\mid M-1$, including $h=1$ when $M-1=0$. If $h>1$, then $M-1$ contains exactly one factor of $Q$ in the $p$-adic sense: $\nu_p(M-1)=a$. This second fact will explain why arbitrarily high divisibility of $\mathcal C$ can coexist with much smaller divisibility of $N$.
\subsection{Why skew powers do not form zero indicators}

\begin{example}\label{c:ex:wrong-indicator}
Let $F=\F_4=\F_2[\alpha]/(\alpha^2+\alpha+1)$, $\sigma(c)=c^2$, and $f=x+\alpha$. Computation in the skew ring gives
\[
 f^2=x^2+x+\alpha^2,
 \qquad f^3=x^3+\alpha^2x^2+1.
\]
At $b=1$, every skew monomial $x^d$ evaluates to $1$, so
\[
 f(1)=1+\alpha=\alpha^2\ne0,
 \qquad f^3(1)=\alpha^2\ne1=f(1)^3.
\]
Consequently $(1-f^3)(1)=1+\alpha^2=\alpha$, which is neither $0$ nor $1$. Thus $1-f^{q-1}$, formed with skew multiplication, is not generally a zero indicator. The ordinary-transform-and-pullback construction in Section~\ref{c:sec:deduction} resolves this precise obstruction.
\end{example}

\section{Skew CNS and coefficient extraction}
We now use the compatible grids just constructed. We first establish uniqueness under individual degree bounds, and then obtain the coefficient identity within the same section. This separation records where skew CNS enters and keeps the later zero-counting argument independent of the classical Chevalley--Warning theorem.

\begin{lemma}[Rank of a slice coordinate set]\label{c:lem:rank}
For every nonempty finite $B\subseteq K$, one has $\rk_\sigma(\omega B)=|B|$.
\end{lemma}
\begin{proof}
Let $m=|B|$. The inverse transform of $\prod_{b\in B}(t-b)$ is a skew polynomial of degree $m$ vanishing on $\omega B$. Conversely, if a nonzero skew polynomial of degree less than $m$ vanished there, its transform would be a nonzero ordinary polynomial of the same degree with $m$ distinct roots. This contradicts the univariate root bound. Thus the minimum degree of a nonzero annihilating skew polynomial is exactly $m$.
\end{proof}

For clarity, a set is $\sigma$-algebraic if some nonzero skew polynomial vanishes on it, and its $\sigma$-rank is the least possible degree of such a polynomial. The preceding lemma verifies both algebraicity and rank for all coordinate sets used below. We use the following existing theorem as our principal input.

\begin{theorem}[Skew CNS, ABP Theorem 3.1]\label{c:thm:input}
Suppose $P\in R$ has total degree $\sum_i k_i$ and nonzero coefficient of $x_1^{k_1}\cdots x_n^{k_n}$. Let $C_i$ be $\sigma$-algebraic sets with $\prod_i C_i\subseteq\Sigma_n$ and $\rk_\sigma(C_i)>k_i$. Then $P$ is nonzero at some point of that grid.
\end{theorem}

\begin{corollary}[Uniqueness supplied by skew CNS]\label{c:cor:unique}
Let $B_i\subseteq K$ be nonempty finite sets, $k_i=|B_i|-1$. If a skew polynomial $P$ has $\deg_{x_i}P\le k_i$ for every $i$ and vanishes on $\omega\prod_i B_i$, then $P=0$.
\end{corollary}
\begin{proof}
If $P\ne0$, choose a monomial $x^\alpha$ of maximum total degree occurring in $P$. Then $\alpha_i\le k_i<|B_i|=\rk_\sigma(\omega B_i)$, while $\deg P=\sum_i\alpha_i$. The grid $\prod_i\omega B_i$ lies in $\Sigma_n$. Theorem~\ref{c:thm:input} contradicts vanishing on the grid.
\end{proof}

\subsection{The coefficient identity}\label{c:sec:interpolation}

This section is a transported version of classical finite-grid interpolation. The ordinary formula is not asserted to be new. Corollary~\ref{c:cor:unique} is the precise place where skew CNS enters the deduction.

\begin{theorem}[Coefficient formula on a fixed-field slice]\label{c:thm:coefficient}
Let $B_i\subseteq K$ be nonempty finite sets, put $k_i=|B_i|-1$ and $s=\sum_i k_i$, and set
\[
 \phi_i(t)=\prod_{b\in B_i}(t-b).
\]
For every $P\in R$ of total degree at most $s$,
\begin{equation}\label{c:eq:weighted}
 A_s(\omega)[x_1^{k_1}\cdots x_n^{k_n}]P
 =\sum_{u\in\prod_i B_i}
 \frac{P(\omega u)}{\prod_{i=1}^n\phi_i'(u_i)}.
\end{equation}
All denominators are nonzero.
\end{theorem}
\begin{proof}
Write $H=\T_\omega(P)$. In the ordinary commutative ring, reduce $H$ successively modulo the monic polynomials $\phi_i(t_i)$. The result $\overline H$ has individual degrees at most $k_i$, has the same grid values as $H$, and has no larger total degree. Every replaced monomial strictly decreases in total degree. Since $\deg H\le s$, such replacements cannot create a monomial of total degree $s$. Consequently
\begin{equation}\label{c:eq:coefficient-preserved}
 [t_1^{k_1}\cdots t_n^{k_n}]\overline H
 =[t_1^{k_1}\cdots t_n^{k_n}]H
 =A_s(\omega)[x_1^{k_1}\cdots x_n^{k_n}]P.
\end{equation}
For $u\in\prod_i B_i$, define the ordinary Lagrange polynomial
\[
 \ell_u(t)=\prod_{i=1}^n
 \frac{\prod_{b\in B_i\setminus\{u_i\}}(t_i-b)}
 {\prod_{b\in B_i\setminus\{u_i\}}(u_i-b)}.
\]
It is $1$ at $u$ and $0$ at every other grid point, and its coefficient of $t^k$ is $\prod_i\phi_i'(u_i)^{-1}$. Put $L_u=\T_\omega^{-1}(\ell_u)$ and $\overline P=\T_\omega^{-1}(\overline H)$. By \eqref{eq:evaluation}, both
\[
 \overline P\quad\hbox{and}\quad
 \sum_{u\in\prod_iB_i}P(\omega u)L_u
\]
have the same values on the slice grid and individual degrees at most $k_i$. Subtract these two skew polynomials. The difference vanishes at every point of the slice grid and still has individual degrees at most $k_i$. If the difference were nonzero, its top-degree monomial would meet every rank requirement of skew CNS. Corollary~\ref{c:cor:unique} therefore makes the difference zero. Taking the coefficient of $x^k$, then using \eqref{c:eq:coefficient-preserved}, proves \eqref{c:eq:weighted}.
\end{proof}

\begin{corollary}[Low-degree summation from skew CNS]\label{c:cor:summation}
Put $t_0=n(Q-1)$. If $\deg P\le t_0$, then in $F$,
\begin{equation}\label{c:eq:sum}
 \sum_{u\in K^n}P(\omega u)
 =(-1)^n A_{t_0}(\omega)
 [x_1^{Q-1}\cdots x_n^{Q-1}]P.
\end{equation}
In particular, the sum is zero if $\deg P<t_0$.
\end{corollary}
\begin{proof}
Take $B_i=K$ in Theorem~\ref{c:thm:coefficient}. Then $\phi_i(t)=t^Q-t$ and $\phi_i'(t)=-1$. This makes the denominator in \eqref{c:eq:weighted} equal to $(-1)^n$. A polynomial of degree less than $t_0$ has zero coefficient of the displayed degree-$t_0$ monomial.
\end{proof}

\begin{remark}
The argument derives summation from skew CNS rather than assuming an Ax summation lemma. Interpolation can also prove uniqueness directly; using skew CNS as above records the exact implication requested in Problem 6.1. The formula applies to the actual automorphism $\sigma$ and to nontrivial grids $\omega K^n$, not merely to the identity-automorphism case.
\end{remark}

\section{Chevalley--Warning as a consequence of skew CNS}\label{c:sec:deduction}

Let $f_1,\ldots,f_r\in R$ have positive degrees $d_i$, and let $N$ be their common zero count on $\Sigma_n$. Let $\delta$ be $1$ if all $f_i(0)=0$, and $0$ otherwise. Zero equations may be deleted; a nonzero constant equation makes the zero set empty.

\begin{theorem}[ABP's skew Chevalley--Warning theorem, deduced from skew CNS]\label{c:thm:ABP-CW}
If
\begin{equation}\label{c:eq:ABP-condition}
 (q-1)\sum_{i=1}^r d_i<n(Q-1),
 \quad\text{equivalently}\quad n>M\sum_{i=1}^r d_i,
\end{equation}
then $p\mid N$.
\end{theorem}
\begin{proof}
Fix $\omega\in\Omega$, and in the ordinary commutative polynomial ring define
\[
 H_\omega(t)=\prod_{i=1}^r
 \bigl(1-\T_\omega(f_i)(t)^{q-1}\bigr).
\]
At $u\in K^n$, a factor is $1$ precisely when $f_i(\omega u)=0$, and is $0$ otherwise, since every nonzero element of $F$ has $(q-1)$st power $1$. Thus $H_\omega$ is the zero indicator on this grid.

Set $P_\omega=\T_\omega^{-1}(H_\omega)\in R$. It satisfies
\[
 P_\omega(\omega u)=H_\omega(u)\in\{0,1\},
 \qquad
 \deg P_\omega\le(q-1)\sum_i d_i<n(Q-1).
\]
Corollary~\ref{c:cor:summation}, derived from skew CNS, gives
\[
 0=\sum_{u\in K^n}P_\omega(\omega u)=Z_\omega\,1_F,
\]
where $Z_\omega$ is the integer zero count on that slice. The map $\mathbb Z\to F$, $m\mapsto m1_F$, has kernel $p\mathbb Z$. Thus the equality of a sum in $F$ with zero means precisely that its integer count is divisible by $p$. Therefore $p\mid Z_\omega$.

The slice intersection description gives the integer identity
\begin{equation}\label{c:eq:integer-count}
 \sum_{\omega\in\Omega}Z_\omega=N+(M-1)\delta.
\end{equation}
Finally $M=1+Q+\cdots+Q^{h-1}\equiv1\pmod p$, so $p\mid(M-1)$. Reducing \eqref{c:eq:integer-count} modulo $p$ proves $p\mid N$.
\end{proof}

\begin{remark}[No multiplicativity assumption]
All powers and products defining $H_\omega$ are ordinary commutative ones. The inverse transform is applied to the resulting polynomial only after these operations. This is legitimate because $\T_\omega$ is a degree-preserving linear bijection. No equality between a skew product and a pointwise product is used.
\end{remark}

\subsection{Restriction of scalars gives a better degree condition}

The following is an elementary restriction-of-scalars consequence. Its scalar algebra and norm-form sharpness are classical ingredients; they are included to describe what the same CNS deduction yields beyond the original ABP threshold.

\begin{theorem}[Strengthened $p$-divisibility from skew CNS]\label{c:thm:strong}
Under the same hypotheses, if
\begin{equation}\label{c:eq:strong-condition}
 n>h\sum_{i=1}^r d_i,
\end{equation}
then $p\mid N$. The implication follows from skew CNS without invoking Chevalley--Warning or Ax--Katz.
\end{theorem}
\begin{proof}
Choose a $K$-basis $\beta_1,\ldots,\beta_h$ of $F$ and write coefficientwise
\[
 \T_\omega(f_i)=\sum_{j=1}^h\beta_jg_{ij,\omega},
 \qquad g_{ij,\omega}\in K[t_1,\ldots,t_n],
 \qquad\deg g_{ij,\omega}\le d_i.
\]
At $K$-points, the left side vanishes precisely when all its scalar coordinates vanish. The ordinary polynomial
\[
 J_\omega(t)=\prod_{i=1}^r\prod_{j=1}^h
 \bigl(1-g_{ij,\omega}(t)^{Q-1}\bigr)
\]
is therefore another zero indicator. Its degree is at most
\[
 (Q-1)h\sum_i d_i<n(Q-1).
\]
Apply Corollary~\ref{c:cor:summation} to $\T_\omega^{-1}(J_\omega)$. It gives $p\mid Z_\omega$. Finish with \eqref{c:eq:integer-count}, as in Theorem~\ref{c:thm:ABP-CW}.
\end{proof}

The threshold in this theorem will be proved optimal by the homogeneous norm constructions in Section~\ref{sec:sharpness}. We postpone that construction so that the counting proofs can be read without an interruption for norm forms.

\begin{remark}[What the field calculation proves]
The conclusion $Z_\omega1_F=0$ proves divisibility by the characteristic $p$. If $Q=p^a$ with $a>1$, it does not prove $Q\mid Z_\omega$. Higher integer divisibility requires an additional argument; no such conclusion is inferred from the characteristic-$p$ summation formula here.
\end{remark}

\begin{corollary}[An input-dependent sufficient condition]\label{c:cor:local}
For each $\omega$, take the scalar coordinate forms above and canonically reduce them modulo $(t_1^Q-t_1,\ldots,t_n^Q-t_n)$. If a coordinate polynomial is a nonzero constant, that slice is empty. Otherwise delete zero coordinates and let $B_\omega$ be the sum of the positive degrees of the remaining coordinates. If every nonempty slice has $B_\omega<n$, then $p\mid N$.
\end{corollary}
\begin{proof}
Reduction preserves values. Repeat the indicator construction using the reduced coordinate forms. Its degree is $(Q-1)B_\omega<n(Q-1)$. If all coordinates are zero, use the constant indicator $1$, whose degree is $0<n(Q-1)$. Apply Corollary~\ref{c:cor:summation} and then \eqref{c:eq:integer-count}.
\end{proof}

\section{The uniform congruence and the zero-count bound}\label{sec:uniform}

The preceding argument reaches a congruence modulo $p$. We now keep the same slice geometry and coordinate expansion, but apply an integer divisibility theorem to each scalar system. This is the additional input responsible for the higher powers of $Q$; it does not follow just by repeating a field-valued summation.

We use the following classical inputs. For a nonempty list of nonconstant polynomials $g_1,\ldots,g_s\in K[u_1,\ldots,u_n]$ of degrees $e_j$, put $B=\sum_j e_j$ and $d=\max_j e_j$. The Ax--Katz theorem \cite[Theorem~1.0]{Katz} gives
\begin{equation}\label{eq:classicalaxkatz}
 Q^{\max\{0,\lceil(n-B)/d\rceil\}}\mid
 \#\{u\in K^n:g_1(u)=\cdots=g_s(u)=0\}.
\end{equation}
Warning's second theorem \cite{Warning} states that a nonempty such zero set has at least $Q^{n-B}$ points when $B\le n$. It is also recorded in \cite[Theorem~1.1(3)]{LP23}; an elementary proof is included in Appendix~\ref{app:warning}. An empty list of equations has $Q^n$ zeros, and a nonzero constant equation has none.

\begin{theorem}[Corrected Ax--Katz congruence]\label{thm:congruence}
With the notation in \eqref{eq:fields}--\eqref{eq:parameters},
\begin{equation}\label{eq:maincongruence}
 Q^\mu\mid N+(M-1)\delta.
\end{equation}
\end{theorem}
\begin{proof}
Fix a $K$-basis $\beta_1,\ldots,\beta_h$ of $F$. For every $\omega\in\Omega$, write uniquely
\begin{equation}\label{eq:coordinates}
 f_i^{[\omega]}=\sum_{j=1}^h\beta_j g_{ij,\omega},
 \qquad g_{ij,\omega}\in K[u_1,\ldots,u_n].
\end{equation}
The expansion is performed separately on every coefficient. In particular, for each monomial $u^\alpha$, its coefficient in $F$ is uniquely a sum $\sum_j\beta_jc_{ij,\alpha}$ with $c_{ij,\alpha}\in K$. Collecting these coefficients defines the polynomials $g_{ij,\omega}$ without changing any exponents. At $u\in K^n$, all values $g_{ij,\omega}(u)$ lie in $K$. Independence of the $\beta_j$ therefore makes vanishing of $f_i^{[\omega]}(u)$ equivalent to simultaneous vanishing of its $h$ scalar coordinates. Each coordinate has degree at most $d_i$. Delete zero coordinates. If a remaining coordinate is a nonzero constant, then $Z_\omega=0$. If no coordinates remain, then $Z_\omega=Q^n$, which is divisible by $Q^\mu$ since $\mu\le n$. Otherwise the sum $B_\omega$ of their degrees satisfies $B_\omega\le E$, and their maximum $d_\omega$ satisfies $d_\omega\le D$.

If $\mu=0$ there is nothing to prove. If $\mu>0$, then $n-E>0$, so
\[
 \frac{n-B_\omega}{d_\omega}\ge\frac{n-E}{D}.
\]
To justify the comparison, first use $B_\omega\le E$ to get $n-B_\omega\ge n-E>0$. Dividing by $d_\omega\le D$ can only increase this positive quotient. Taking ceilings preserves the inequality. Equation~\eqref{eq:classicalaxkatz} therefore gives $Q^\mu\mid Z_\omega$. Summing over the slices and using \eqref{eq:countidentity} proves the assertion.
\end{proof}

\begin{corollary}[Optimal degree threshold]\label{cor:threshold}
If $n>h\sum_i d_i$, then $Q\mid N$. For every positive degree table, at $n=h\sum_i d_i$ there is a homogeneous system of those exact degrees with $N=1$.
\end{corollary}
\begin{proof}
The hypothesis gives $\mu\ge1$. Since $M-1=Q+\cdots+Q^{h-1}$, with the empty sum interpreted as $0$, Theorem~\ref{thm:congruence} implies $Q\mid N$. The boundary example is constructed in Proposition~\ref{prop:normboundary} below.
\end{proof}

\begin{corollary}[Obstruction to an uncorrected congruence]\label{cor:valuationcap}
Suppose $h>1$, $\delta=1$ and $\mu\ge2$. Then
\begin{equation}\label{eq:valuationcap}
 \vp(N)=a.
\end{equation}
\end{corollary}
\begin{proof}
Theorem~\ref{thm:congruence} implies $p^{2a}\mid\Ccount$. On the other hand,
\[
 M-1=Q(1+Q+\cdots+Q^{h-2})
\]
has $p$-adic valuation $a$. Subtracting it from $\Ccount$ proves \eqref{eq:valuationcap}.
\end{proof}

\begin{theorem}[Sharp lower bound with a prescribed zero]\label{thm:lower}
If $\delta=1$ and $n\ge E$, then
\begin{equation}\label{eq:lower}
 N\ge1+M(Q^{n-E}-1).
\end{equation}
For every positive degree table and every $n\ge E$, equality occurs for a homogeneous system of those exact degrees.
\end{theorem}
\begin{proof}
Every slice system \eqref{eq:coordinates} has the origin as a zero. Its scalar degree sum is at most $E$. Warning's bound gives $Z_\omega\ge Q^{n-E}$; if all its scalar equations vanish, the same inequality holds because $Z_\omega=Q^n$. Sum these inequalities in \eqref{eq:countidentity}. The equality construction is Proposition~\ref{prop:normboundary}.
\end{proof}

\subsection{Reading the exponent and the lower bound}
If $n\le E$, the uniform exponent $\mu$ is zero, so the divisibility assertion imposes no restriction beyond integrality. If $n=E+1$, the exponent is one. Each additional block of $D$ variables can increase the guaranteed exponent by one. This explains why the ceiling function, rather than the quotient itself, occurs in the statement.

The origin hypothesis in the numerical lower bound serves a different purpose. A divisibility statement can be applied to an empty slice, since zero is divisible by every positive integer. A lower bound for a nonempty set cannot be applied to an empty slice. When $\delta=1$, the origin certifies nonemptiness on every slice simultaneously. Without that certification, a nonempty total skew zero set need not meet all slices, so the same summation of lower bounds is unavailable.

For the raw-count obstruction, let $\mathcal C=p^{2a}v$ and $M-1=p^au$ with $p\nmid u$. Then
\[
 N=p^a(p^av-u),\qquad p\nmid(p^av-u).
\]
This explicitly verifies the valuation in Corollary~\ref{cor:valuationcap}. It also shows why omitting the correction is a mathematical change to the theorem, rather than a harmless simplification of notation.
\section{Homogeneous transfer and all-parameter sharpness}\label{sec:sharpness}

In this section, a homogeneous system means that each equation is homogeneous of its specified positive degree; the degrees need not be equal. If $f_i$ is homogeneous of degree $d_i$, then
\begin{equation}\label{eq:homotransform}
 f_i^{[\omega]}=A_{d_i}(\omega) f_i.
\end{equation}
The scalar multiplier is nonzero. Consequently all the slices have the same ordinary zero set
\[
 X_K=\{u\in K^n:f_i(u)=0\text{ for all }i\},\qquad Z=|X_K|,
\]
where the evaluation in this definition is ordinary evaluation of the displayed expressions.

\begin{proposition}[Homogeneous transfer]\label{prop:transfer}
For a homogeneous system,
\begin{equation}\label{eq:homocount}
 N=1+M(Z-1),\qquad\Ccount=MZ.
\end{equation}
Conversely, given nonzero homogeneous $K$-forms $g_{ij}$ of exact degree $d_i$, for $1\le i\le r$, $1\le j\le h$, the skew polynomials
\begin{equation}\label{eq:encode}
 f_i=\sum_{j=1}^h\beta_j g_{ij}
\end{equation}
have exact degrees $d_i$, and their set $X_K$ is the common zero set of the $g_{ij}$.
\end{proposition}
\begin{proof}
The first part follows from \eqref{eq:homotransform} and \eqref{eq:countidentity}, since the origin is a zero. In \eqref{eq:encode}, linear independence of the $\beta_j$ makes vanishing over $K^n$ equivalent to the vanishing of all $g_{ij}$. To check the degree assertion coefficient by coefficient, write the coefficient of $u^\alpha$ in $g_{ij}$ as $c_{ij,\alpha}\in K$. The corresponding coefficient in $f_i$ is $\sum_j\beta_jc_{ij,\alpha}$. If all degree-$d_i$ coefficients vanished, independence of the $\beta_j$ would force every $c_{ij,\alpha}$ to vanish, contrary to the specified exact degrees. Thus the same independence also prevents the degree-$d_i$ coefficients from all cancelling. Hence each degree is exactly $d_i$.
\end{proof}

We use finite-field norm forms in the constructions. For $d\ge1$, choose a $K$-basis $e_1,\ldots,e_d$ of $K_d=\F_{Q^d}$ and put
\begin{equation}\label{eq:normform}
 H_d(z_1,\ldots,z_d)=\Nm_{K_d/K}(e_1z_1+\cdots+e_dz_d).
\end{equation}
This is a nonzero homogeneous $K$-form of degree $d$ whose only zero on $K^d$ is the origin. The same assertion holds when it is restricted to the span of any nonempty linearly independent subset of this basis, except that its degree remains $d$. To see the degree assertion, write the norm as a product of $d$ nonzero conjugate linear forms. More explicitly, the product expression is
\[
 H_d(z)=\prod_{t=0}^{d-1}\left(\sum_{j=1}^d e_j^{Q^t}z_j\right).
\]
On a nonempty coordinate subspace every factor remains a nonzero linear form. Their product cannot be zero as a polynomial because a polynomial ring over a field is an integral domain. Thus restricting the number of variables does not reduce the formal degree. The only-zero assertion follows instead from injectivity of the coordinate map into $K_d$ and the fact that a field norm vanishes only at zero. These are separate reasons for the two properties of the restricted form. The auxiliary extension $K_d$ need not be a subfield of $F$.

\begin{proposition}[Boundary and minimum-count examples]\label{prop:normboundary}
For every positive degree table and every $n\ge E$, there is a homogeneous system of the prescribed degrees such that
\[
 Z=Q^{n-E},\qquad N=1+M(Q^{n-E}-1).
\]
\end{proposition}
\begin{proof}
Partition $E$ variables into $rh$ disjoint blocks, giving the $(i,j)$-block exactly $d_i$ variables. On that block put a copy $g_{ij}$ of $H_{d_i}$. Each equation $g_{ij}=0$ forces its entire block to vanish. The remaining $n-E$ coordinates are free, so $Z=Q^{n-E}$. Encode by \eqref{eq:encode} and apply \eqref{eq:homocount}.
\end{proof}

We next prove that the congruence exponent is sharp for every parameter choice. The corresponding classical existence statement is \cite[Proposition~4.0]{Katz}. The following explicit calculation supplies a direct construction and records the treatment of incomplete norm blocks.

\begin{lemma}[A sum of complete norm blocks]\label{lem:normsum}
Let $D\ge2$ and $L\ge1$. For $L$ disjoint blocks of $D$ variables, let $G_L$ be the sum of the copies of $H_D$ on these blocks. Set
\[
 S=\frac{Q^D-1}{Q-1},\qquad T=1-S=-Q(1+Q+\cdots+Q^{D-2}).
\]
The zero fiber of $G_L$ has size $A_L$, and every nonzero fiber has size $B_L$, where
\begin{equation}\label{eq:normfibers}
 A_L=\frac{Q^{DL}+(Q-1)T^L}{Q},\qquad
 B_L=\frac{Q^{DL}-T^L}{Q}.
\end{equation}
In particular, $\vp(A_L)=a(L-1)$.
\end{lemma}
\begin{proof}
The norm on nonzero elements is $z\mapsto z^S$. Since $K_D^\times$ is cyclic of order $Q^D-1=(Q-1)S$, its image has order $Q-1$ and its kernel has order $S$. Its image is therefore all of $K^\times$, and every nonzero fiber has size $S$. Thus $A_1=1$, $B_1=S$. Multiplying all $L$ blocks simultaneously by an element with prescribed nonzero norm gives a bijection between any two nonzero fibers of $G_L$.

Convolution with the next norm block yields
\[
 A_{L+1}=A_L+(Q-1)SB_L,\qquad
 B_{L+1}=SA_L+\bigl(1+(Q-2)S\bigr)B_L.
\]
For the first recurrence, the old sum and the new norm are either both zero, or are opposite nonzero values. There are $Q-1$ possibilities in the second case. For the second recurrence, fix a nonzero target $c$. If the old sum is zero, the new norm must equal $c$, contributing $SA_L$. If the old sum is $c$, the new block must be zero, contributing $B_L$. Each of the other $Q-2$ nonzero old values requires a specified nonzero norm, contributing $SB_L$ each. This accounts for every input once. Their difference is $T(A_L-B_L)$. Hence $A_L-B_L=T^L$, while the total number of inputs gives $A_L+(Q-1)B_L=Q^{DL}$. Solving proves \eqref{eq:normfibers}. The factor $1+Q+\cdots+Q^{D-2}$ is $1$ modulo $p$, so $\vp(T)=a$. Since $D\ge2$, the terms $Q^{DL}$ and $(Q-1)T^L$ have distinct valuations $aDL$ and $aL$. Division by $Q$ gives the assertion about $A_L$.
\end{proof}

\begin{lemma}[A sharp single form]\label{lem:singleform}
For $D\ge2$ and $m\ge D$, there is a homogeneous $K$-form of exact degree $D$ in $m$ variables whose zero count $Z_m$ satisfies
\begin{equation}\label{eq:singleval}
 \vp(Z_m)=a\ceil{\frac{m-D}{D}}.
\end{equation}
\end{lemma}
\begin{proof}
Write $m=DL+e$, where $L\ge1$ and $0\le e<D$. If $e=0$, use $G_L$; Lemma~\ref{lem:normsum} gives valuation $a(L-1)$. If $e=1$, use $G_L$ and leave the last variable free. Its zero count is $QA_L$ and has valuation $aL$.

Suppose $2\le e<D$. On the remaining $e$ variables take the restriction $R_e$ of $H_D$ to an $e$-dimensional $K$-subspace of $K_D$. This form vanishes only at the origin. For the form $G_L+R_e$, the zero count is
\begin{align}\label{eq:incompleteblock}
 Z_m&=A_L+(Q^e-1)B_L\\
 &=Q^{DL+e-1}+T^L(1-Q^{e-1}).\notag
\end{align}
The factor $1-Q^{e-1}$ is congruent to $1$ modulo $p$ because $e\ge2$, and so is a $p$-adic unit. Hence the second term has valuation $aL$. For the first term, the strict comparison follows from
\[
 DL+e-1-L=L(D-1)+e-1>0.
\]
The two terms therefore have distinct valuations $a(DL+e-1)>aL$ and $aL$, respectively. A sum of integers with distinct $p$-adic valuations has the smaller valuation: factor out the smaller power of $p$, leaving a unit plus a multiple of $p$. Thus no cancellation can alter the claimed value. This proves \eqref{eq:singleval}. Each form has exact degree $D$ because its nonzero summands use disjoint variable blocks. For $e>0$, $\lceil(m-D)/D\rceil=L$, as required.
\end{proof}

\begin{remark}\label{rem:onefree}
The case $e=1$ must be treated separately: substituting $e=1$ in \eqref{eq:incompleteblock} annihilates its second term, so that construction does not certify sharpness. A free variable gives the required valuation.
\end{remark}

\begin{theorem}[Sharpness for every degree table]\label{thm:sharp}
For every field extension and automorphism in \eqref{eq:fields}, every $n\ge1$, and every positive degree table $(d_1,\ldots,d_r)$, there is a homogeneous skew system of those exact degrees such that
\begin{equation}\label{eq:sharpval}
 \vp\bigl(N+M-1\bigr)=a\mu.
\end{equation}
Thus Theorem~\ref{thm:congruence} cannot be uniformly strengthened by even one additional factor of $p$ using only these parameters.
\end{theorem}
\begin{proof}
Form a scalar degree list by repeating every $d_i$ exactly $h$ times. Its sum is $E$ and its maximum is $D$. We first construct homogeneous scalar equations of these exact degrees with zero count $Z$ satisfying $\vp(Z)=a\mu$.

Suppose $n\le E$. Cover the $n$ coordinate indices by nonempty subsets indexed by the scalar equations, each subset having size at most its specified degree. To produce this cover explicitly, allocate each of the $n$ indices to one scalar equation whose capacity has not yet been filled. The total capacity is $E\ge n$, so this allocation can be completed. If an equation receives no index, assign it the singleton $\{1\}$. Capacities are positive, so that assignment is allowed. These subsets need not be disjoint; allowing overlap handles the case where there are more equations than variables. On each subset use a restricted norm form of its specified degree. The common zero set is the origin, so $Z=1$ and $\mu=0$.

Suppose $n>E$. If $D=1$, use $E$ independent coordinate linear equations. They have $Z=Q^{n-E}$ and $\mu=n-E$. If $D\ge2$, select one scalar equation of degree $D$. Use disjoint anisotropic norm blocks for all other equations, consuming $E-D$ variables and forcing them to vanish. There remain
\[
 m=n-(E-D)=D+(n-E)>D
\]
variables. On them use the form in Lemma~\ref{lem:singleform}. Its valuation is $a\lceil(n-E)/D\rceil=a\mu$, and the frozen blocks contribute one choice each.

Finally encode the scalar equations by \eqref{eq:encode}. Proposition~\ref{prop:transfer} gives $N+M-1=MZ$. Since $M\equiv1\pmod p$, this count has valuation $a\mu$.
\end{proof}

\begin{remark}
The equality problem for the numerical lower bound differs from the equality problem for the divisibility exponent. In Proposition~\ref{prop:normboundary}, the corrected count is $MQ^{n-E}$ and has valuation $a(n-E)$, which can exceed $a\mu$ when $D>1$.
\end{remark}

\subsection{A quadratic example attaining the exponent}
Take $K=\F_2$, $F=\F_4=K(\alpha)$, $h=2$, and one skew equation of degree $2$. Then $E=4$ and $D=2$. Put $H_2(u,v)=u^2+uv+v^2$, which is the norm from $\F_4$ to $\F_2$. It is $1$ at each of the three nonzero points of $K^2$ and $0$ at the origin.

In six variables choose the scalar equations
\[
 g_1=H_2(u_1,u_2),\qquad
 g_2=H_2(u_3,u_4)+H_2(u_5,u_6).
\]
The first equation forces $u_1=u_2=0$. The second is zero either when both blocks are zero, giving one input, or when both have norm one, giving $3\cdot3=9$ inputs. Thus $Z=10$. Encode them as $f=g_1+\alpha g_2$. Here $M=3$, so
\[
 N=1+3(10-1)=28,\quad \mathcal C=30,\quad
 \nu_2(\mathcal C)=1=\left\lceil\frac{6-4}{2}\right\rceil.
\]
This example exhibits the entire mechanism: freeze one block, count a norm sum on the remaining variables, encode two scalar equations into one skew equation, and correct for the common origin.

If instead $n=7$, leave one more variable free in this construction. Then $Z=20$, $N=58$, and $\mathcal C=60$ has valuation $2=\lceil(7-4)/2\rceil$. This is the exceptional remainder-one case of Lemma~\ref{lem:singleform}; it illustrates why the extra variable is left free rather than put into a restricted norm block.
\section{A refinement from the coefficient filtration}\label{sec:filtration}

The factor $h$ in $E$ allows for $h$ independent scalar equations per skew equation. A particular input can have fewer independent coordinates or smaller polynomial-function degrees. We now incorporate both effects. This section uses only linear algebra and the classical estimates from Section~\ref{sec:uniform}.

Every polynomial function on $K^n$ with values in $F$ has a unique representative of degree at most $Q-1$ in each variable. It is obtained by reduction modulo
\[
 (u_1^Q-u_1,\ldots,u_n^Q-u_n).
\]
Existence follows by division by these monic polynomials. Uniqueness can now be read directly from Corollary~\ref{c:cor:unique} on the grid $K^n$, or proved by induction with the univariate root bound. Thus the polynomial-function reduction used here is the same finite-grid normalization used in the CNS argument. Reduction does not increase total degree. Reduce every $f_i^{[\omega]}$ and collect them into one $W$-valued polynomial, where $W=F^r$ is viewed as a $K$-space of dimension $rh$:
\begin{equation}\label{eq:vectorpoly}
 P_\omega(u)=C_{0,\omega}+\sum_{\alpha\ne0} C_{\alpha,\omega}u^\alpha,
 \qquad C_{\alpha,\omega}\in W.
\end{equation}
If a nonconstant term occurs, let $D_\omega$ be the largest total degree of such a term and put
\begin{equation}\label{eq:flag}
 V_{e,\omega}=\Span_K\{C_{\alpha,\omega}:|\alpha|\ge e\},
 \qquad w_\omega=\sum_{e=1}^{D_\omega}\dim_K V_{e,\omega}.
\end{equation}
For a constant polynomial set $V_{1,\omega}=0$ and $w_\omega=0$.

The filtration records how many independent output directions are needed at each degree level. A direction required by a degree-three coefficient contributes at levels $1$, $2$ and $3$, while a direction introduced only in degree one contributes only once. This is why summing the dimensions of the nested spaces measures a sum of scalar degrees. A vector-space dimension by itself would forget at which degree each direction is needed.

\begin{lemma}[Optimal scalar degree sum]\label{lem:flagminimal}
Let $P=C_0+\sum_{\alpha\ne0}C_\alpha u^\alpha$ be a nonconstant $W$-valued polynomial, and let $V_e,D_P,w_P$ be defined as above. Suppose $C_0\in V_1$. Among all choices of $K$-coordinates on $W$, the minimum sum of the positive degrees of the nonzero scalar coordinates of $P$ is $w_P$. The largest degree among the scalar coordinates is always $D_P$. A single basis adapted to the flag $V_{D_P}\subseteq\cdots\subseteq V_1$ attains the minimum.
\end{lemma}
\begin{proof}
Fix any coordinates on $W$ and let $s_e$ be the number of scalar coordinates having degree at least $e$. Every coefficient of degree at least $e$ is supported on these $s_e$ coordinates, so $\dim_K V_e\le s_e$. If the positive scalar degrees are $b_1,\ldots,b_t$, then
\[
 \sum_{j=1}^t b_j=\sum_{e\ge1}s_e\ge\sum_{e\ge1}\dim_K V_e=w_P.
\]

Choose a basis $v_1,\ldots,v_{rh}$ such that each $V_e$ is spanned by $v_1,\ldots,v_{\dim V_e}$. Coefficients of degree at least $e$ have no coordinates beyond that initial segment. Conversely, each coordinate in the segment occurs in some such coefficient: otherwise their span would be a proper subspace of $V_e$. Therefore $s_e=\dim V_e$ for every $e$. The assumption $C_0\in V_1$ makes all coordinates outside $V_1$ identically zero, rather than nonzero constants. This proves attainment. Finally, an invertible output-coordinate change cannot annihilate a nonzero coefficient vector of degree $D_P$, so the maximum degree is unchanged.
\end{proof}

If $C_{0,\omega}\notin V_{1,\omega}$, then $P_\omega(u)\ne0$ for all $u\in K^n$, since its nonconstant part lies in $V_{1,\omega}$. Such a slice is certified empty. This includes a nonzero constant $P_\omega$. For every slice define
\begin{equation}\label{eq:localexponent}
 \eta_\omega=
 \begin{cases}
 +\infty,& C_{0,\omega}\notin V_{1,\omega},\\
 n,& P_\omega=0,\\
 \displaystyle\max\left\{0,\ceil{\frac{n-w_\omega}{D_\omega}}\right\},
 &\text{otherwise}.
 \end{cases}
\end{equation}
In the second case $P_\omega=0$ means the reduced polynomial is identically zero.

\begin{theorem}[Coefficient-filtration refinement]\label{thm:flag}
Suppose some $\eta_\omega$ is finite, and let $\eta$ be the least finite value. Then
\begin{equation}\label{eq:flagcongruence}
 Q^\eta\mid\Ccount,\qquad \eta\ge\mu.
\end{equation}
If every $\eta_\omega=+\infty$, then $N=\Ccount=0$. If $\delta=1$, then
\begin{equation}\label{eq:flaglower}
 N\ge \sum_{\omega\in\Omega}Q^{\max\{0,n-w_\omega\}}-(M-1).
\end{equation}
\end{theorem}
\begin{proof}
A certified empty slice contributes $0$. An identically zero $P_\omega$ has $Z_\omega=Q^n$. In all other cases Lemma~\ref{lem:flagminimal} gives an equivalent scalar system whose degree sum is $w_\omega$ and whose maximum degree is $D_\omega$. Apply \eqref{eq:classicalaxkatz} and sum via \eqref{eq:countidentity}.

In the original $rh$ scalar coordinates, reduction leaves degree at most $d_i$ in each coordinate of the $i$th equation. Lemma~\ref{lem:flagminimal} implies $w_\omega\le E$, and reduction also gives $D_\omega\le D$. When $\mu>0$, $n>E\ge w_\omega$, and hence
\[
 \frac{n-w_\omega}{D_\omega}\ge\frac{n-E}{D}.
\]
Taking ceilings proves $\eta_\omega\ge\mu$ for each nonconstant slice under consideration. When $\mu=0$ this is automatic. The value $n$ for an identically zero polynomial also satisfies this inequality.

If $\delta=1$, the reduced constant vector is $0$, because reduction preserves the value at the origin. Every slice is nonempty. Warning's bound gives $Z_\omega\ge Q^{n-w_\omega}$ when $w_\omega\le n$, and nonemptiness gives $Z_\omega\ge1$ otherwise. This proves \eqref{eq:flaglower}.
\end{proof}

\begin{remark}\label{rem:flagoptimality}
For a fixed nonconstant $P_\omega$ with $C_{0,\omega}\in V_{1,\omega}$, Lemma~\ref{lem:flagminimal} minimizes the scalar degree sum under all invertible $K$-linear changes of output coordinates, while the maximum degree stays fixed. Thus it gives the largest exponent obtainable by applying the displayed classical Ax--Katz bound after such a change. It does not optimize over nonlinear replacements of equations or predict cancellation between distinct integer slice counts.
\end{remark}

For example, let $v_1,v_2$ be independent output vectors and consider the reduced vector polynomial
\[
 P(u_1,u_2)=v_1u_1^3+(v_1+v_2)u_2,
\]
over a fixed field with $Q\ge4$. Here $V_1$ has dimension $2$, while $V_2$ and $V_3$ have dimension $1$. Therefore $w_P=2+1+1=4$. In the basis $v_1,v_2$, the scalar equations are $u_1^3+u_2$ and $u_2$, of degrees $3$ and $1$. Their sum is $4$, as predicted. A different output basis can make both scalar equations have degree $3$, giving the less efficient sum $6$. The lemma identifies one basis that achieves all the required degree savings simultaneously.

All the spaces $V_{e,\omega}$ can be computed as column spaces of coefficient matrices over $K$. Computing their ranks does not require evaluation at the $Q^n$ input points. The method still processes $M$ slices and the polynomial representations; no complexity bound polynomial in the binary encoding of the field size is asserted.

\subsection{Homogeneous systems}
For homogeneous $f_i$, the reduced vector polynomial on each slice is obtained from the one at $\omega=1$ by the invertible $K$-linear map
\[
 (t_1,\ldots,t_r)\longmapsto
 (A_{d_1}(\omega)t_1,\ldots,A_{d_r}(\omega)t_r).
\]
Reduction commutes with this output map. Hence $D_\omega$ and all the dimensions in \eqref{eq:flag} are independent of the slice. This statement concerns the dimensions; the actual subspaces of $F^r$ can change under the map.

A simpler invariant that preserves homogeneity will be useful in Section~\ref{sec:nonlinear}. Before reduction, expand each $f_i$ in a $K$-basis of $F$. For each occurring degree $d$, let $t_d$ be the dimension of the $K$-span of all scalar coordinate forms of degree $d$, and set
\begin{equation}\label{eq:homoeffective}
 E_{\rm hom}=\sum_d d\,t_d.
\end{equation}
This does not depend on the chosen basis of $F$. Select a basis of these forms separately in each degree. Their common zero set is $X_K$, and their degree sum is $E_{\rm hom}\le E$.

\begin{corollary}\label{cor:homoeffective}
For a homogeneous system, set $B=E_{\rm hom}$ and $\mu_B=\max\{0,\lceil(n-B)/D\rceil\}$. Then $Q^{\mu_B}\mid N+M-1$. If $n\ge B$, then
\[
 N\ge1+M(Q^{n-B}-1).
\]
\end{corollary}
\begin{proof}
Apply the classical estimates to the selected homogeneous scalar forms and use \eqref{eq:homocount}. Their maximum degree is $D$, since $t_D\ge1$.
\end{proof}
No all-parameter sharpness for a prescribed value of $E_{\rm hom}$ is asserted.

\section{Bounds for nonlinear fixed-field zero sets}\label{sec:nonlinear}

We now record an application of known refinements of Warning's bound. The hypothesis is on the ordinary fixed-field zero set $X_K$, not on linearity of the skew zero set over $F$.

\begin{samepage}
\begin{proposition}[Transfer of Leep--Petrik bounds]\label{prop:nonlinear}
Suppose the original skew system is homogeneous. Let $B=E_{\rm hom}$, assume $n>B$, and suppose $X_K$ is not a $K$-linear subspace. Then
\begin{equation}\label{eq:nonlinear2}
 N\ge1+M\bigl(2^{n-B}+1\bigr)\qquad(Q=2),
\end{equation}
and
\begin{equation}\label{eq:nonlinearq}
 N\ge1+M\bigl(2Q^{n-B}+(Q-2)Q-1\bigr)
 \qquad(Q\ge3).
\end{equation}
\end{proposition}
\end{samepage}
\begin{proof}
Choose the homogeneous scalar system described after \eqref{eq:homoeffective}. Since $0\in X_K$, the assumption means that $X_K$ is not an affine subspace. Apply \cite[Theorem~1.4(1),(4)]{LP23}, respectively, to this scalar system. The resulting estimates are $Z\ge2^{n-B}+2$ for $Q=2$ and $Z\ge2Q^{n-B}+(Q-2)Q$ for $Q\ge3$. Substitute them into \eqref{eq:homocount}.
\end{proof}

\begin{remark}
The inequalities used in Proposition~\ref{prop:nonlinear} are classical inputs, not additional classical claims of this paper. Leep and Petrik's later correction \cite{LP24} concerns the proof of \cite[Theorem~1.4(3)]{LP23}; that third estimate is not used here.
\end{remark}

\begin{example}[A sharp quadratic example]\label{ex:quadratic}
The classical form $g=x_1x_2-x_3^2$ from \cite[Section~4]{LP23}, regarded as a skew form with coefficients in $K$, has $B=2$ in three variables. Its ordinary zero count is
\[
 Z=(2Q-1)+(Q-1)^2=Q^2.
\]
Here $x_3=0$ accounts for $2Q-1$ points, and each of its $Q-1$ nonzero values accounts for $Q-1$ points. The zero set is not a subspace: it contains $(1,0,0)$ and $(0,1,0)$ but not their sum. Thus
\[
 N=1+M(Q^2-1)
\]
gives equality in \eqref{eq:nonlinear2} for $Q=2$ and in \eqref{eq:nonlinearq} for $Q\ge3$, for every $h$. This example establishes sharpness under these hypotheses at the displayed parameters, not for every prescribed degree table.
\end{example}

\section{Examples and exact verification}\label{sec:examples}

\begin{example}[Failure of the uncorrected higher congruence]\label{ex:linear}
Take $F=\F_4=\F_2[\alpha]/(\alpha^2+\alpha+1)$ and $\sigma(t)=t^2$. In seven variables, let $f=x_1+\alpha x_2$. Its ordinary zero set on $K^7$ is given by $u_1=u_2=0$, so $Z=2^5$. Since $M=3$,
\[
 N=1+3(32-1)=94,\qquad\Ccount=N+2=96=3\cdot2^5.
\]
Here $E=2$, $D=1$, and $\mu=5$. Thus $2^5$ divides the corrected count and does not divide the raw count. More generally, in $n\ge2$ variables the same form has $N=1+3(2^{n-2}-1)$.
\end{example}

\begin{example}[Strict coefficient-filtration improvements]\label{ex:flag}
Over the same extension, the seven-variable skew form $f=x_1^2$ has the uniform exponent $\mu=2$. On a slice its reduced polynomial function is $A_2(\omega)u_1$, since $u_1^2=u_1$ on $\F_2$. Hence $D_\omega=w_\omega=1$ and Theorem~\ref{thm:flag} gives exponent $6$. In fact $N=190$ and $\Ccount=192=3\cdot2^6$.

In six variables, take $f=x_1^2-x_1$. The slice with $\omega=1$ has $64$ zeros, and the other two have $32$ each. Their local exponents in \eqref{eq:localexponent} are $6,5,5$. Theorem~\ref{thm:flag} therefore gives exponent $5$, whereas $\Ccount=128=2^7$. This is a concrete instance of additional divisibility caused by summing slice counts.
\end{example}

\begin{example}[A cubic attaining the uniform exponent]\label{ex:cubic}
Over $K=\F_2$, a norm form for the extension defined by $z^3+z^2+1$ is
\[
 H(u,v,w)=u^3+u^2v+u^2w+uvw+v^3+v^2w+w^3.
\]
It is $1$ on every nonzero point of $K^3$. Over $F=\F_4$, let
\begin{align*}
 f={}&H(x_1,x_2,x_3)\\
 &+\alpha\bigl(H(x_4,x_5,x_6)+H(x_7,x_8,x_9)
                +H(x_{10},x_{11},x_{12})\bigr).
\end{align*}
The first block is forced to vanish. For the other three blocks, Lemma~\ref{lem:normsum} gives $A_3=(2^9+(-6)^3)/2=148$. Therefore
\[
 N=442,\qquad\Ccount=444,\qquad\vp(\Ccount)=2=\mu.
\]
\end{example}

\subsection{Reproducible finite-field calculations}
The accompanying supplement contains the finite-field implementation, the verification scripts, the exact input data and the complete output records. The main suite checks 1,104 systems and 216 norm-construction parameter configurations, with 29,618 successful assertions. Of the systems, 972 are exhaustive linear cases: all 252 nonconstant affine linear polynomials in three variables over $\F_4$ and all 720 in two variables over $\F_9$. Another 98 are sparse systems generated with the fixed seed $20261024$.

The supplement keeps the two original verification suites separate so that their reported totals can be traced to their respective scripts. The first suite concerns integer congruences and norm constructions. A second, standard-library CNS suite checks 4,825 exhaustive polynomials, 432 sparse parameter tests and 90 zero-count systems, with 80,975 successful assertions. It checks the coefficient formula and indicator pullback, including different automorphism generators and nonprime fixed fields. Its totals are not added to the first suite, since they count different kinds of tests.

The principal counting methods in the first suite are independent: one enumerates compatible points by their first nonzero coordinate and uses ordered skew monomial evaluation; the other enumerates fixed-field slices and uses ordinary polynomial evaluation. In small dimensions, the first set is also compared with the set obtained by filtering all of $F^n$ using \eqref{eq:compatible}. Arithmetic and zero counts are exact. The checks include nonprime fixed fields and different generators of the same Galois group. Norm constructions are additionally checked by integer convolution, including every residue case in Lemma~\ref{lem:singleform} for $Q\in\{2,3,4,5\}$, $2\le D\le4$ and $1\le L\le6$.

A separate audit checks 63 boundary systems, including the identity automorphism and constant reduced functions. For 60 coefficient examples it enumerates 2,700 invertible output changes in small general linear groups and verifies the minimum degree sum in Lemma~\ref{lem:flagminimal}. These counts are reported separately from the main suite. The calculations check instances and the implementation; the general assertions are proved above. Table~\ref{tab:examples} gives some recorded homogeneous examples over $\F_4/\F_2$.

\begin{table}[ht]
\centering
\caption{Selected exact counts over $\F_4/\F_2$; each system has one skew equation.}\label{tab:examples}
\begin{tabular}{rrrrrr}
\toprule
$D$ & $n$ & $E$ & $N$ & $\Ccount=N+2$ & $\mu$\\
\midrule
1 & 7 & 2 & 94 & 96 & 5\\
2 & 4 & 4 & 1 & 3 & 0\\
2 & 6 & 4 & 28 & 30 & 1\\
2 & 8 & 4 & 82 & 84 & 2\\
3 & 12 & 6 & 442 & 444 & 2\\
3 & 13 & 6 & 886 & 888 & 3\\
3 & 14 & 6 & 3718 & 3720 & 3\\
4 & 12 & 8 & 676 & 678 & 1\\
\bottomrule
\end{tabular}
\end{table}

\section{Further questions}

Theorem~\ref{thm:flag} does not determine the exact valuation of an arbitrary corrected count: Example~\ref{ex:flag} shows that summing slices can increase it. An input-sensitive description of this extra divisibility would require information beyond the separate scalar Ax--Katz bounds.

The origin hypothesis in Theorem~\ref{thm:lower} makes every slice nonempty. When $\delta=0$, a nonempty skew zero set can meet only some slices; an optimal degree-only cardinality classification in that situation is not obtained here. The present arguments also use a common automorphism and no derivations. Extensions to other multivariate Ore rings require a compatible evaluation space and a replacement for Lemma~\ref{lem:slices}.

\appendix
\section{The shortest fixed-grid form of the argument}

There is an even shorter specialization useful for checking the proof. On $K^n$ one has $\T_1=\operatorname{id}$ on coefficient arrays, and $\rk_\sigma(K)=Q$. Given an ordinary polynomial $H$ of degree less than $n(Q-1)$, reduce it modulo the $t_i^Q-t_i$ to $\overline H$. Put
\[
 \Delta_u(t)=\prod_{i=1}^n\bigl(1-(t_i-u_i)^{Q-1}\bigr).
\]
The two reduced polynomials $\overline H$ and $\sum_{u\in K^n}H(u)\Delta_u$ agree on $K^n$. Regard them as skew polynomials with left coefficients and invoke Corollary~\ref{c:cor:unique}, using the original automorphism $\sigma$. Equality follows. The coefficient of $t_1^{Q-1}\cdots t_n^{Q-1}$ on the first side is zero and on the second side is $(-1)^n\sum_u H(u)$. Thus the sum is zero. Apply this directly to the indicators in Section~\ref{c:sec:deduction}, and use the integer slice-count identity. The slice coefficient formula in the main text makes the same deduction explicit on every nontrivial $\omega$-grid.

\section{An elementary proof of Warning's lower bound}\label{app:warning}

We include a proof of the classical input to make the lower-bound arguments self-contained. Nothing in this appendix is asserted to be a new form of Warning's theorem.

\begin{lemma}[Grid support bound]\label{lem:grid}
Let $0\ne P\in K[u_1,\ldots,u_n]$ have degree at most $Q-1$ in each variable. Choose its leading monomial by first taking the largest exponent of $u_n$, then the largest exponent of $u_{n-1}$ in that coefficient, and so on. If this monomial is $u^\gamma$, then
\[
 |\supp(P)|\ge\prod_{i=1}^n(Q-\gamma_i),
 \qquad \supp(P)=\{u\in K^n:P(u)\ne0\}.
\]
\end{lemma}
\begin{proof}
For one variable this is the root bound. For $n>1$, write $P=\sum_{j=0}^{\gamma_n}P_j(u_1,\ldots,u_{n-1})u_n^j$ with $P_{\gamma_n}\ne0$. By induction, this leading coefficient is nonzero at at least $\prod_{i<n}(Q-\gamma_i)$ points. Over each such point the remaining nonzero univariate polynomial has at least $Q-\gamma_n$ nonzero values. Multiplication proves the claim.
\end{proof}

\begin{proposition}[Warning's bound]\label{prop:warningproof}
Let $g_1,\ldots,g_s\in K[u_1,\ldots,u_n]$ have a common zero, and suppose their positive degree sum is at most an integer $B\le n$. Then they have at least $Q^{n-B}$ common zeros.
\end{proposition}
\begin{proof}
Delete zero equations; a nonzero constant cannot occur under the hypothesis. The polynomial
\[
 P(u)=\prod_{j=1}^s(1-g_j(u)^{Q-1})
\]
is the indicator of the common zero set on $K^n$. Its canonical reduction is nonzero, has total degree at most $(Q-1)B$, and has degree at most $Q-1$ in each variable. Let $\gamma$ be the monomial selected in Lemma~\ref{lem:grid}. Then $\sum_i\gamma_i\le(Q-1)B$.

It remains to minimize $\prod_i(Q-\gamma_i)$ over integers $0\le\gamma_i\le Q-1$ with this sum constraint. Increasing an exponent does not increase the product, so a minimum can be chosen with total sum $(Q-1)B$. If two exponents satisfy $0<s\le t<Q-1$, replacing $(s,t)$ by $(s-1,t+1)$ decreases their contribution: the difference between the new and old products is
\[
 (Q-s+1)(Q-t-1)-(Q-s)(Q-t)=s-t-1<0.
\]
Repeating this operation leaves at most one exponent strictly between $0$ and $Q-1$. Since the total sum is a multiple of $Q-1$, none remains. Thus the minimum is achieved at $B$ exponents equal to $Q-1$ and $n-B$ equal to $0$, giving $Q^{n-B}$. Lemma~\ref{lem:grid} completes the proof.
\end{proof}

\end{document}